\documentclass{amsart}

\usepackage{hyperref}
\usepackage{graphicx} % Required for inserting images
\usepackage{amsthm}
\usepackage{amsmath}
\usepackage{amssymb}
\usepackage{verbatim}
\usepackage[colorinlistoftodos]{todonotes}

\newtheorem{defi}{Definition}[section]
\newtheorem{theorem}[defi]{Theorem}
\newtheorem{fact}[defi]{Fact}
\newtheorem{remark}[defi]{Remark}
\newtheorem{lemma}[defi]{Lemma}
\newtheorem{corollary}[defi]{Corollary}
\newtheorem{proposition}[defi]{Proposition}

\newcommand{\Ind}{
 \setbox0=\hbox{$x$}\kern\wd0\hbox to 0pt{\hss$
 \mid$\hss}\lower.9\ht0\hbox to 0pt{\hss$\smile$\hss}\kern\wd0
}
\newcommand{\indep}[3]{
 #1\mathop{\mathpalette\Ind{}}_{#2}#3
}

\newcommand{\Notind}{
 \setbox0=\hbox{$x$}\kern\wd0\hbox to 0pt{\mathchardef
 \nn=12854\hss$\nn$\kern1.4\wd0\hss}\hbox to 0pt{\hss$\mid$\hss}\lower.9\ht0
 \hbox to 0pt{\hss$\smile$\hss}\kern\wd0
}
\newcommand{\depen}[3]{
 #1\mathop{\mathpalette\Notind{}}_{#2}#3
}

\newcommand{\acl}{\mathrm{acl}}
\newcommand{\dcl}{\mathrm{dcl}}
\newcommand{\tp}{\mathrm{tp}}

\title{Ampleness in the Farey graph}
\author{Zahra Mohammadi Khangheshlaghi \& Rizos Sklinos}
\date{\today}
\thanks{\\ The first-named author was funded by the Deutsche Forschungsgemeinschaft under Germany's Excellence Strategy EXC 2044/2--390685587, Mathematics M\"unster: Dynamics--Geometry--Structure. She gratefully acknowledges the hospitality of the Academy of Mathematics and Systems Science, Chinese Academy of Sciences, where part of this work was carried out. \\ \\
The second-named author was supported by the National Natural Science Foundation of China (NSFC), grant no.~12350610234.}

\begin{document}

\begin{abstract}
We show that the first-order theory of the Farey graph is $1$-ample but not $2$-ample. Along the way, we prove that it weakly eliminates imaginaries and observe that forking is trivial. 
\end{abstract}

\maketitle

\section{Introduction}
The ample hierarchy is a way to measure how complicated the geometry of forking independence is in a stable theory. ``Abelian structures", e.g. Abelian groups and modules, on the one hand are considered the simplest and are placed at the first-level of this hierarchy, while infinite fields have maximum complexity \cite{PillayGeomFork}. The hierarchy is particularly meaningful as it has been shown that it does not collapse \cite{Tent2014Ample, BaudischMartinPizarroZiegler2014Ample}, i.e. for each $n<\omega$, there exists a first-order theory $T_n$ such that $T_n$ is $n$-ample but not $n+1$-ample. Nevertheless, the model-theoretic content of ampleness remains only partially understood.

The starting point for the present work is provided by the $\omega$-stability of the theory of the Farey graph \cite{Farey_Graph}, \cite{DiSarloKoberdaDeLaNuez2023}. In the present paper we will use the recent understanding of forking independence and of algebraic closure in the above mentioned theory in order to prove: 

\begin{theorem}
The first-order theory of the Farey graph is CM-trivial, i.e. not $2$-ample. 
\end{theorem}

On the way to proving the above result we show that the theory of the Farey graph weakly eliminates imaginaries. In addition, we observe that forking independence is trivial, thus no infinite group is interpretable in it. 

\section{Preliminaries}

\subsection{Ampleness} 
Formally the ample hierarchy is defined as follows:

\begin{defi}\label{Ample}\cite{PillayGeomFork, AmpleDividing}
Let $T$ be a stable theory and $n\geq1$. Then $T$ is $n$-ample if (after possibly adding some parameters) there are 
$a_0,a_1,\ldots,a_n$ such that:
\begin{enumerate}
\item $a_0$ forks with $a_n$ over $\emptyset$;
\item $a_{i+1}$ does not fork with $a_0,\ldots,a_{i-1}$ over $a_i$, for $1\leq i<n$;
\item $\acl^{eq}(a_0)\cap \acl^{eq}(a_1)=\acl^{eq}(\emptyset)$;
\item $\acl^{eq}(a_0,\ldots,a_{i-1},a_i)\cap acl^{eq}(a_0,\ldots,a_{i-1},a_{i+1})=\acl^{eq}(a_0,\ldots,a_{i-1})$, for $1\leq i<n$.
\end{enumerate}

\end{defi}

To understand how ample a theory is one first needs to understand imaginaries. An imaginary can be thought of as canonical code assigned to a definable set. These codes exist in a mild expansion, $T^{eq}$, of the theory $T$. When these codes exist in the original theory we say that the theory admits elimination of imaginaries (see \cite{Poizat83c}).     

\begin{defi}
A (complete) first-order theory $T$ (strongly) eliminates imaginaries if for every $e$ in $\mathbb{M}^{eq}$ there exists a finite tuple $\bar a\in\mathbb{M}$, such that $e\in \dcl^{eq}(\bar a)$ and $\bar a\in \dcl^{eq}(e)$. Equivalently, every imaginary element is interdefinable with a real tuple.

We say that T weakly eliminates imaginaries if, for every $e\in \mathbb M^{eq}$, there is a finite tuple $\bar a\in \mathbb M$ such that $e\in \dcl^{eq}(\bar a)$ and 
$\bar a\in \acl^{eq}(e)$.

\end{defi}

The theories $ACF_p$ of algebraically closed fields are standard examples of theories with elimination of imaginaries \cite{Poizat83c}. The theory of infinite sets weakly eliminates imaginaries, but does not (strongly) eliminate imaginaries. 

We note in passing the following fact. 

\begin{fact}
A first-order theory $T$ weakly eliminates imaginaries if and only if every definable set has a smallest algebraically closed set of parameters it is defined over. 
\end{fact}

We remark that the condition that any definable set has a smallest definably closed set of parameters it is defined over is strictly weaker than (strong) elimination of imaginaries. 

In practice the following lemma is very useful for proving weak elimination of imaginaries (see \cite[Lemma 16.17]{PoizatBookModelTheory}). 

\begin{lemma}
Let $T$ be a first-order theory. Suppose that the following conditions hold:
\begin{itemize}
    \item there is no strictly descending infinite chain of algebraically
    closed sets generated by finite sets,
    \[
    \acl(A_0)\supsetneq \acl(A_1)\supsetneq\ldots\supsetneq \acl(A_n) \supsetneq \ldots 
    \]
    \item if $A$ and $B$ are algebraically closed sets generated by finite
    sets and $X$ is definable over both $A$ and $B$, then $X$ is definable
    over $A\cap B$.
\end{itemize}
Then $T$ weakly eliminates imaginaries.
\end{lemma}
\begin{proof}
As a consequence of the first condition, for any definable set $X$ there exists a minimal algebraically closed set over which $X$ is definable. The second condition implies that there exists a smallest one.  
\end{proof}

\subsection{Some model theory of the Farey Graph} 
The Farey graph, denoted by $FG$, is the graph whose vertex set is the set of extended rational numbers $\widehat{\mathbb Q}:=\mathbb Q\cup\{\infty\}$.
We represent each vertex by a reduced fraction $p/q$, with $p,q\in\mathbb Z$, $\gcd(p,q)=1$, and adopt the convention that $\infty=1/0$. Two distinct vertices $p/q$ and $r/s$ are joined by an edge precisely when $|ps-qr|=1$. We regard $FG$ as a structure in the language consisting of a single symmetric and irreflexive binary relation, interpreted as adjacency.

The following point of view is closer to our purposes. Through the identification
$Out(\mathbb F_2)\cong GL_2(\mathbb Z)$, the Farey graph carries a natural action
of $Out(\mathbb F_2)$. Its vertices may be identified with conjugacy classes of
rank-one free factors of $\mathbb F_2$, and two vertices are adjacent precisely
when the corresponding free factors admit generators forming a basis of
$\mathbb F_2$. This viewpoint is closely related to the rank-two case of
Culler--Vogtmann Outer space and, more generally, to the free factor complex
\cite{HatcherVogtmann1998}, which will be the subject of a subsequent paper.

Let $Th(FG)$ be the first-order theory of the Farey graph. The models of this theory can be described simply by gluing Farey graphs together. Two Farey graphs are glued along a single vertex and globally the combinatorics of the gluing are described by a forest. More formally.

\begin{fact}\cite{Farey_Graph}
Let $\mathcal{G}=(V(\mathcal{G}),E(\mathcal{G}))$ be a nonempty forest. For each
$v\in V(\mathcal{G})$, let
$
FG_v=(X_v,R_v)
$
be a copy of the Farey graph, and assume that the sets $\{X_v:v\in V(\mathcal{G})\}$ are pairwise disjoint. 

For every edge $e=\{u,v\}\in E(\mathcal{G})$, choose vertices
$x_{e,u}\in X_u
\ \text{and} \ 
x_{e,v}\in X_v.$
Let $X:=\bigsqcup_{v\in V(\mathcal{G})}X_v$,
and let $\sim$ be the smallest equivalence relation on $X$ satisfying $x_{e,u}\sim x_{e,v}$
for every edge $e=\{u,v\}\in E(\mathcal{G})$. 

Define a graph
$
FG(\mathcal{G})=(X/{\sim},R)
$
by declaring that, for $[x],[y]\in X/{\sim}$,
$
R([x],[y])
$
holds if and only if there exist $v\in V(\mathcal{G})$ and
$x',y'\in X_v$ such that
$
x'\sim x, 
y'\sim y, \text{and} \ 
R_v(x',y').
$
Then $FG(\mathcal{G})\models Th(FG)$.

Conversely, if $M\models Th(FG)$, then there exists a forest $\mathcal{G}$ such that

$$
M\cong FG(\mathcal{G}).
$$

\end{fact}

It follows from the description of $1$-types that a model $M$ of $Th(FG)$ is $\kappa$-saturated if and only if $M$ is isomorphic to $FG(\mathcal{G})$ with $\mathcal{G}$ having at least $\kappa$ connected components and each $a\in M$ belongs to at least $\kappa$ distinct Farey graphs.

The forest $\mathcal G$ appearing in the preceding construction is not
canonically determined by the resulting model: non-isomorphic forests
may give rise to isomorphic models. For instance, a configuration of four
Farey graphs meeting at a single vertex can be represented either by a
path on four vertices or by a three-rayed star, provided that all edge
identifications are made at the same distinguished vertex. To remove this ambiguity, \cite{Farey_Graph} associates to every model a
canonical bipartite incidence forest, whose two classes of vertices are
the vertices of the model and its Farey subgraphs. 

\begin{defi}
Let $M\models Th(FG)$, and let $\mathcal F(M)$ denote the set of Farey
graphs in $M$. We define $G_{\mathrm{tree}}(M)$ to be the bipartite graph
with bipartition $M\sqcup\mathcal F(M)$, where $a\in M$ is adjacent to $F\in\mathcal F(M)$ precisely when $a\in F$.
\end{defi}

\begin{fact}\label{fact:Gtree}
Let $M\models Th(FG)$. Then $G_{\mathrm{tree}}(M)$ is a forest canonically
associated with $M$.
\end{fact}

From now on, all graph-theoretic notions, including paths, geodesics,
convexity, gates, and connected components, are understood with respect
to the canonical incidence forest $G_{\mathrm{tree}}(M)$.

We now recall the description of algebraic closure from \cite{Farey_Graph}. 

\begin{fact}[Algebraic closure]\label{FGAlc}
Let $M\models Th(FG)$ and suppose first that
$G_{\mathrm{tree}}(M)$ is connected. For $A\subseteq M$, let
$Conv_{\mathrm{tree}}(A)$ denote the convex hull of $A$ in $G_{\mathrm{tree}}(M)$, and set
\[
Conv_{M}(A)
=
\bigcup_{F\in Conv_{\mathrm{tree}}(A)\cap\mathcal F(M)}F.
\]
Then
\[
\acl_M(A)
=
\begin{cases}
\emptyset, & \text{if } A=\emptyset,\\[1mm]
A, & \text{if } |A|=1,\\[1mm]
Conv_{M}(A), & \text{if } |A|\geq 2.
\end{cases}
\]

In general, write
\[
M=\bigsqcup_{i\in I}M_i,
\]
where the $M_i$ are the connected components of $M$, and put
$A_i=A\cap M_i$. Then
\[
\acl_M(A)
=
\bigsqcup_{i\in I}\acl_{M_i}(A_i).
\]
\end{fact}

%Let $M\models Th(FG)$ and $A\subset M$. Suppose $A=\sqcup_{i\in I} A_i$, where each $A_i$ is a subset of a distinct connected component of $M$. Then $\acl_M(A)$    
%
%\end{fact}

Finally we recall the description of forking independence from \cite{Farey_Graph}. We adopt the convention that the convex hull of the empty set is empty. 

\begin{fact}[Forking independence]\label{FGFind}
Let $M\models Th(FG)$ be a sufficiently saturated model, and let
$A,B,C\subseteq M$. Then $B$ is independent from $C$ over $A$ if and
only if every geodesic in $G_{\mathrm{tree}}(M)$ joining a point of
$\acl(B)$ to a point of $\acl(C)$ meets $Conv_{\textrm{tree}}(A)$.

In particular, $B$ is independent from $C$ over the empty set if and
only if no connected component of $M$ contains both an element of $B$
and an element of $C$.
\end{fact}

An easy consequence of the above fact is that any type in this theory is stationary, equivalently $\dcl^{eq}(\emptyset)=\acl^{eq}(\emptyset)$. 

\begin{corollary}\label{StationaryEmpty}
Every type $p(\bar x)\in S_n(Th(FG))$ is stationary.
\end{corollary}

\begin{proof}
Let $p(\bar x)=\tp(\bar a/\emptyset)$, and let $A$ be an arbitrary set of parameters. Suppose that $\bar b$ realizes a non-forking extension of $p$ to $A$. Then $\tp(\bar b/\emptyset)=\tp(\bar a/\emptyset)$,
and, since $\indep{\bar b}{\emptyset}{A}$, every connected component containing an element of $\bar b$ is disjoint from the connected components meeting $A$.

Now let $\bar c$ be another realization of a non-forking extension of $p$ to $A$. Since both $\bar b$ and $\bar c$ realize $p$, there is an automorphism sending $\bar b$ to $\bar c$. Moreover, this automorphism may be chosen to act non-trivially only on the connected components disjoint from $A$. Consequently, it fixes $A$ pointwise, and therefore $\tp(\bar b/A)=\tp(\bar c/A)$.

Hence the non-forking extension of $p$ to $A$ is unique, so $p$ is stationary.
\end{proof}

\begin{remark}
Types over arbitrary parameter sets need not be stationary. Let $a,b$ be adjacent vertices in the Farey graph, and let $c,d$ denote the two common neighbours of $a$ and $b$. Since there is an automorphism fixing $a$ and $b$ and interchanging $c$ and $d$, we have $\tp(c/a,b)=\tp(d/a,b)$. However, $c,d\in\acl(a,b)$, so $\tp(c/\acl(a,b))\neq\tp(d/\acl(a,b))$.

As algebraic extensions are non-forking, these are two distinct non-forking extensions of $\tp(c/a,b)$, showing that $\tp(c/a,b)$ is not stationary.
\end{remark}

On the other hand, the preceding phenomenon cannot occur over algebraically closed parameter sets. The following result naturally extends the previous corollary from the empty set to arbitrary algebraically closed parameter sets.

\begin{corollary}\label{StationAlgClosed}
Let $A=\acl(A)\subseteq\mathbb M$. Then every type over $A$ is stationary.
\end{corollary}
\begin{proof}
It is enough to prove that any $1$-type, $p(x)$, over $A$ is stationary. Let $B\supset A$ and $tp(a_1/B)$, $tp(a_2/B)$ be non-forking extensions of $p(x)=tp(a/A)$. By Corollary \ref{StationaryEmpty}, we may assume that $A$ has non-empty
intersection with the connected component containing $a_1$. Since
$\tp(a_1/A)=\tp(a_2/A)$, it follows that $a_2$ also belongs to the same connected component. Without loss of generality we may assume that $a_1, a_2$ and $B$ belong to the same connected component of $\mathbb{M}$.

By the description of forking independence, every path from $a_i$ to
$B$ passes through $A$. Since $A$ is convex, there is a unique element
$g_i\in A$, called the \emph{gate} of $A$ with respect to $a_i$, such
that every path from $a_i$ to a point of $A$ passes through $g_i$.
Since $\tp(a_1/A)=\tp(a_2/A)$, we have $g_1=g_2=:g$. Note that $g$ is always of element type. 

Let $T_i$ denote the connected component of
$G_{\mathrm{tree}}(\mathbb{M})\setminus\{g\}$ containing $a_i$, and let
$\sigma\in Aut(\mathbb{M}/A)$ satisfy $\sigma(a_1)=a_2$. We note that $\sigma$ induces a unique automorphism of $G_{\mathrm{tree}}(\mathbb{M})$ and since $\sigma(g)=g$, it follows that $\sigma(T_1)=T_2$. If $T_1\neq T_2$, let
$f$ agree with $\sigma$ on $T_1$, with $\sigma^{-1}$ on $T_2$, and with
the identity on every other connected component of
$G_{\mathrm{tree}}(\mathbb{M})\setminus\{g\}$. If $T_1=T_2$, let $f$ agree with $\sigma$ on
$T_1$ and with the identity on every other connected component. In both
cases, $f$ is an automorphism of $\mathbb{M}$, since the connected
components of $G_{\mathrm{tree}}(\mathbb{M})\setminus\{g\}$ are disjoint, and 
$g$ is fixed.

Finally, since every path from $a_i$ to $B$ passes through $g$, we have that the components that contain $B$ are disjoint from $T_1$ and $T_2$. Hence $f$ fixes $B$ pointwise and sends $a_1$ to $a_2$. Therefore $tp(a_1/B)=tp(a_2/B)$.
\end{proof}

Finally, for what follows it will be convenient to work in the language
in which quantifier elimination for $Th(FG)$ was established. We first
recall the relevant terminology. A vertex $v$ of a subgraph of the Farey
graph is called \emph{removable} if either $v$ has valency at most one,
or $v$ has valency two and its two neighbours are adjacent.

\begin{fact}[{\cite{Farey_Graph}}]\label{FGQE}
A simple cycle $C$ in the Farey graph is called minimal if it has exactly
two removable vertices. If $x$ and $y$ are its removable vertices, we say
that $C$ is minimal for $x,y$. For $n\geq 1$, let $\mathcal C_n$ be the
set of isomorphism types of minimal triangulated cycles whose removable
vertices have distance $n$ in the cycle, and put
$\mathcal C=\bigcup_{n\geq 1}\mathcal C_n$.

For $C\in\mathcal C$, let $P_C(x,y)$ hold if there is a triangulated
cycle isomorphic to $C$ which is minimal for $x,y$. More generally, let
$\delta=(C_1,\ldots,C_m)$, where $C_i\in\mathcal C_{k_i}$. The binary
predicate $P_\delta(x,y)$ holds if there are connecting points
$z_0=x,z_1,\ldots,z_m=y$ such that $P_{C_i}(z_{i-1},z_i)$ holds for
$i=1,\ldots,m$ and
\[
d(x,y)=\sum_{i=1}^m k_i.
\]
For the empty sequence $\delta$, we put $P_\delta(x,y)$ if and only if
$x=y$.

For $n\geq1$, fix an enumeration
$v_1,\ldots,v_{|F_n|}$ of the Farey graph $F_n$ of level $n$. For
$\sigma\in\operatorname{Sym}(|F_n|)$, let $D_{n,\sigma}(x,y,z)$ express
that $x,y,z$ extend to a copy of $F_n$ in which, with respect to the
fixed enumeration, they occupy the positions prescribed by $\sigma$.

Finally, let
$\varepsilon=(\delta_1,\delta_2,\delta_3,n,\sigma)$, where
$\delta_1,\delta_2,\delta_3$ are finite, possibly empty, sequences in
$\mathcal C$. The ternary predicate $Y_\varepsilon(x,y,z)$ holds if there
exist $x',y',z'$ such that
$D_{n,\sigma}(x',y',z')$ holds and
\[
P_{\delta_1}(x,x'),\qquad
P_{\delta_2}(y,y'),\qquad
P_{\delta_3}(z,z')
\]
hold.

Let $L'$ be the language of graphs expanded by all predicates
$P_\delta$ and $Y_\varepsilon$ as above. Then $Th(FG)$ has quantifier
elimination in $L'$.
\end{fact}

\section{CM-triviality of the Farey Graph}

It is not hard to see that $Th(FG)$ does not eliminate imaginaries. 

\begin{lemma}
The first-order theory of the Farey graph does not eliminate imaginaries. 
\end{lemma}
\begin{proof}
Let $M\models Th(FG)$ be such that $M=M_1\sqcup M_2$,
where $M_1$ and $M_2$ are disjoint copies of the Farey graph. Choose
$a\in M_1$ and $b\in M_2$, and let $X=\{a,b\}$.
Suppose, towards a contradiction, that $X$ is eliminated by a finite
real tuple $\bar c$ from $M$.

We first observe that $\bar c\neq\emptyset$. Indeed, the set $X$ is not
$\emptyset$-definable, since there exists an automorphism of $M$ moving $a$
within $M_1$.

Choose an isomorphism $f\colon M_1\to M_2$ such that $f(a)=b$, and define
$\sigma\in Aut(M)$ by $\sigma|_{M_1}=f$ and $\sigma|_{M_2}=f^{-1}$. Then $\sigma(a)=b$ and $\sigma(b)=a$, and hence $\sigma(X)=X$. 
On the other hand, $\sigma$ exchanges the two disjoint connected
components of $M$ and therefore fixes no real element, hence it moves $\bar c$.
\end{proof}

We will show that it weakly eliminates them, i.e. every definable set has a smallest algebraically closed set over which it is definable. The latter will be a consequence of two propositions. First, we record the following immediate consequence of the description of algebraic closure in $Th(FG)$.

\begin{proposition}\label{DCCAlg}
Let $M$ be a model of $Th(FG)$. Then the algebraic closures of finite subsets of $M$ satisfy the descending chain condition.
\end{proposition}

\begin{proof}
By Fact \ref{FGAlc}, every algebraically closed set $X=\acl(A)$ generated by a finite
subset $A$ of $M$ is a finite disjoint union of finite unions of copies of the
Farey graph and singleton sets.

For such a set $X$, we define the following ordered couple of natural numbers:
\[
\mu(X)=\bigl(f(X),s(X)\bigr)\in\mathbb N^2,
\]
where $f(X)$ is the number of copies of the Farey graph contained in $X$,
and $s(X)$ is the number of connected components of $X$ which are
singletons. We order $\mathbb N^2$ lexicographically.

If $Y\subsetneq X$ and both $X$ and $Y$ are algebraic closures of finite
sets, then either $Y$ contains strictly fewer copies of the Farey graph than
$X$, or it contains the same copies of the Farey graph and strictly fewer
singleton components. Notice that the first possibility includes the case
where one or more copies of the Farey graph are replaced by a singleton.
Hence
\[
\mu(Y)<_{\mathrm{lex}}\mu(X).
\]

Thus every strictly descending chain of algebraic closures of finite sets
induces a strictly descending chain in
\((\mathbb N^2,<_{\mathrm{lex}})\), which is impossible.
\end{proof}

We then only need to show that if a set $X$ is definable over $A$ and over $B$, where $A=acl(\bar a)$ and $B=acl(\bar b)$, then $X$ is definable over their intersection. For this we will be needing some preliminary results. 

%We first observe that for any two distinct elements $a$ and $b$ from a model of $Th(FG)$ we get that $acl(a)\cap acl(b)=\emptyset$. 

We first prove.

\begin{lemma}\label{rmk:dependency}
Let $A=\acl(\bar a)$ and $B=\acl(\bar b)$ be algebraically closed
subsets of a connected component of $\mathbb{M}\models Th(FG)$, with $A\cap B\neq \emptyset$. Then $\indep{A}{A\cap B}{B}$.

Moreover, for any $x\in\mathbb{M}$
\[
\depen{x}{A\cap B}{A}
\quad\Longrightarrow\quad
\indep{x}{A\cap B}{B}.
\]
\end{lemma}
\begin{proof}
Let $a\in A$ and $b\in B$, and choose $d\in A\cap B$. The geodesics
joining $a,b,d$ form a tripod in $G_{\textrm{tree}}(\mathbb{M})$, and let $m$ denote its center. The geodesic joining $a$ with $d$ is contained in $Conv_{\textrm{tree}}(A)$. Similarly, the geodesic joining $b$ with $d$ is contained in $Conv_{\textrm{tree}}(B)$. It follows that $m\in Conv_{\textrm{tree}}(A) \cap Conv_{\textrm{tree}}(B)$, which is $Conv_{\textrm{tree}}(A\cap B)$. On the other hand, by definition of the center of the tripod, $m$ also belongs to the geodesic joining $a$ with $b$. Hence every geodesic joining an element of $A$ to an element of $B$ meets $Conv_{\textrm{tree}}(A\cap B)$. The conclusion now follows from Fact \ref{FGFind}.

For the second assertion, we may assume that the vertex corresponding to $x$ belongs to the same connected component of $G_{\textrm{tree}}(\mathbb{M})$  as $A$ and $B$, since otherwise the conclusion is immediate by Fact~\ref{FGFind}. Let $a\in A$ and $b\in B$. By the first part, the geodesic joining $a$ and $b$ meets $Conv_{\textrm{tree}}(A\cap B)$. On the other hand, since $\depen{x}{A\cap B}{A}$, Fact~\ref{FGFind} implies that there is $a\in A$ such that the geodesic joining $x$ with $a$ is disjoint from  $Conv_{\textrm{tree}}(A\cap B)$.

We fix the witness $a\in A$ of the forking condition and we take $b\in B$ arbitrary. Let $m$ denote the center of the tripod formed by $a,b$ and $x$. Since the geodesic from $a$ to $x$ is disjoint from $Conv_{\textrm{tree}}(A\cap B)$, whereas the geodesic from $a$ to $b$ meets it, the vertex $m$ does not belong to $Conv_{\textrm{tree}}(A\cap B)$. Hence every point of the intersection of the geodesic from $a$ to $b$ with $Conv_{\textrm{tree}}(A\cap B)$ lies on the segment from $m$ to $b$. Since this segment is contained in the geodesic $[b,x]$, the latter geodesic also meets the $Conv_{\textrm{tree}}(A\cap B)$. As $b$ was arbitrary, another application of Fact~\ref{FGFind} yields $\indep{x}{A\cap B}{B}$.      
\end{proof}

We next record the component-wise version of the preceding result.

\begin{corollary} \label{IndAlgClosed}

Let $A=\acl(\bar a)$ and $B=\acl(\bar b)$ be algebraically closed subsets of
$\mathbb{M}\models Th(FG)$. Suppose that every connected component of $G_{\mathrm{tree}}(\mathbb{M})$ meeting both $A$ and $B$ contains a point of $A\cap B$. Then
\[
\indep{A}{A\cap B}{B}.
\]
Moreover, for any $x\in\mathbb{M}$, 
\[
\depen{x}{A\cap B}{A}
\quad\Longrightarrow\quad
\indep{x}{A\cap B}{B}.
\] 
\end{corollary}

We next prove that forking is trivial. 

\begin{proposition}\label{TrivialityForking}
Let $\mathbb{M}\models Th(FG)$, $A\subset \mathbb{M}$ and $a_1, a_2, a_3$ are pairwise independent over $A$. Then $\{a_1, a_2, a_3\}$ is an independent set over $A$.
\end{proposition}

\begin{proof}
We may assume that $a_1,a_2,a_3$ lie in the same connected component
of $G_{\mathrm{tree}}(\mathbb{M})$ and $A$ must intersect this component non-trivially, since otherwise the conclusion follows
immediately from Fact~\ref{FGFind}.

It is enough to prove that $a_1$ is independent from $a_2,a_3$ over
$A$. By pairwise independence, the geodesics joining $a_1$ to $a_2$
and $a_1$ to $a_3$ meet $Conv_{\textrm{tree}}(A)$. 

Let $v\in acl(a_2,a_3)$. By Fact \ref{FGAlc} either $v$ lies on the
geodesic $[a_2,a_3]$, in which case we set $w_v=v$, or there is a
unique Farey-graph type vertex $w_v\in[a_2,a_3]$ such that $v\in w_v$.  Choose $d_2\in [a_1,a_2]\cap Conv_{\mathrm{tree}}(A)$ and
$d_3\in [a_1,a_3]\cap Conv_{\mathrm{tree}}(A)$. Since $Conv_{\mathrm{tree}}(A)$ is convex, the geodesic joining $d_2$ and $d_3$ is contained in it. Moreover, since $w_v\in[a_2,a_3]$, the geodesic joining $a_1$ to $w_v$ meets the geodesic joining $d_2$ to $d_3$. Hence the geodesic joining $a_1$ to $w_v$ meets $Conv_{\mathrm{tree}}(A)$. It follows that the geodesic joining $a_1$ to $v$ meets $Conv_{\mathrm{tree}}(A)$.

Therefore every geodesic from $a_1$ to
$acl(a_2,a_3)$ meets
$Conv_{\mathrm{tree}}(A)$. By
Fact~\ref{FGFind}, $a_1$ is independent from $a_2, a_3$ over $A$.
By symmetry, $\{a_1,a_2,a_3\}$ is independent over $A$.
\end{proof}

Although not needed in the remainder of the paper, we record the following immediate corollary. 

\begin{corollary}
No infinite group is interpretable in $Th(FG)$.
\end{corollary}

%\begin{remark}\cite{Farey_Graph}
%Let $M$ be a model of $T_{FG}$ and $a \in M$. Then $acl(a) = a$ .
%    
%\end{remark}
%Therefore, for every $a, b \in M$, $acl(a) \cap acl(b) = \emptyset$.

%%%    \[ Aut(M/A \cap B)= \neq \langle Aut(M/A), Aut(M/B)\rangle \]

%\begin{remark}\label{rmk:dependency}
%    Suppose that $M$ is a saturated model of $T_{FG}$, $A = acl(\bar{a})$, $B = acl(\bar{b})$ and $A \cap B \neq \emptyset$, $x \in M$.
%    Then if $\depen{x}{A \cap B}{A}$, then 
%    $\indep{x}{A \cap B}{B}$.
%\end{remark}
\begin{comment}
    Prove triviality here!
\end{comment}

\begin{lemma}\label{Intersect}
Let $M\models Th(FG)$ be sufficiently saturated, and let
$A=\acl(\bar a)$ and $B=\acl(\bar b)$ be algebraically closed subsets of
$M$. Suppose that every connected component of
$G_{\mathrm{tree}}(M)$ meeting $A$ or $B$ contains a point of
$A\cap B$. 
Then, for every $n$, every $\bar x\in M^n$, and every
$\sigma\in Aut(M/A\cap B)$, there exists $\delta\in\langle Aut(M/A), Aut(M/B)\rangle$ such that $\delta(\bar x)=\sigma(\bar x)$.
\end{lemma}

\begin{proof}
Let $\bar x\in M^n$ and let $\sigma\in Aut(M/A\cap B)$. Put
$\sigma(\bar x)=\bar y$. By the extension property for non-forking and
the saturation of $M$, we may choose a tuple $\bar y'$ such that
$\tp(\bar y'/A\cap B)=\tp(\bar x/A\cap B)$ and
\[
\indep{\bar y'}{A\cap B}{A\cup B}.
\]

Reordering the coordinates if necessary, we may assume that
$x_i$ forks with $B$ over $A\cap B$ precisely for
$1\leq i\leq m$. By Corollary~\ref{IndAlgClosed}, each $x_i$ is independent from $A$ over
$A\cap B$. Hence, by Proposition~\ref{TrivialityForking},
$x_1,\ldots,x_m$ is independent from $A$ over $A\cap B$.

We next claim that
$x_1,\ldots,x_m$ is independent from
$x_{m+1},\ldots,x_n$ over $A\cap B$. Otherwise, by
Proposition~\ref{TrivialityForking}, there are indices
$1\leq k\leq m$ and $m+1\leq j\leq n$ such that
$x_k$ forks with $x_j$ over $A\cap B$. Thus there is a geodesic
from $x_k$ to $x_j$ avoiding $Conv_{\mathrm{tree}}(A\cap B)$.
Since $x_k$ forks with $B$ over $A\cap B$, there is also a geodesic
from $x_k$ to some element of $B$ avoiding
$Conv_{\mathrm{tree}}(A\cap B)$. As $G_{\mathrm{tree}}(M)$ is a
forest, the geodesic from $x_j$ to this element of $B$ is contained
in the union of the previous two geodesics, and therefore also
avoids $Conv_{\mathrm{tree}}(A\cap B)$. This contradicts the choice
of $x_j$.

The same argument, using the fact that
$\bar y'$ is independent from $B$ over $A\cap B$, shows that
$y'_{m+1},\ldots,y'_n$ is independent from
$x_1,\ldots,x_m$ over $A\cap B$.

Since
$x_{m+1},\ldots,x_n$ is independent from $B$ over $A\cap B$ and
independent from $x_1,\ldots,x_m$ over $A\cap B$,
Proposition~\ref{TrivialityForking} yields that it is independent from
$B\cup\{x_1,\ldots,x_m\}$ over $A\cap B$. Likewise,
$y'_{m+1},\ldots,y'_n$ is independent from
$B\cup\{x_1,\ldots,$ $x_m\}$ over $A\cap B$. Since
$$
\tp(x_{m+1},\ldots,x_n/A\cap B)
=
\tp(y'_{m+1},\ldots,y'_n/A\cap B),
$$
Corollary~\ref{StationAlgClosed} implies that
$$
\tp(x_{m+1},\ldots,x_n/B,x_1,\ldots,x_m)
=
\tp(y'_{m+1},\ldots,y'_n/B,x_1,\ldots,x_m).
$$
Hence there is an automorphism
$\delta_B\in Aut(M/B,x_1,\ldots,x_m)$ sending
$(x_{m+1},\ldots,$ $x_n)$ to
$(y'_{m+1},\ldots,y'_n)$. Since $\delta_B$ fixes $A\cap B$, the tuple
$(x_1,\ldots,x_m,y'_{m+1},\ldots,y'_n)$ has the same type over
$A\cap B$ as $\bar y'$. Moreover, by
Proposition~\ref{TrivialityForking}, it is independent from $A$ over
$A\cap B$, as is $\bar y'$. Applying
Corollary~\ref{StationAlgClosed} once more, we obtain
$$
\tp(x_1,\ldots,x_m,y'_{m+1},\ldots,y'_n/A)
=
\tp(\bar y'/A).
$$
Therefore there is an automorphism
$\delta_A\in Aut(M/A,y'_{m+1},\ldots,y'_n)$ sending
$(x_1,\ldots,$ $x_m)$ to $(y'_1,\ldots,y'_m)$.
Consequently,
$$
\delta_A\circ\delta_B(\bar x)=\bar y'.
$$

Applying the same argument to $\bar y$ and $\bar y'$, we obtain an
element of $\langle Aut(M/A),$ $Aut(M/B)\rangle$ sending $\bar y$ to
$\bar y'$. Taking its inverse and composing with
$\delta_A\circ\delta_B$, we obtain an element of
$\langle Aut(M/A), Aut(M/B)\rangle$ sending $\bar x$ to $\bar y$,
as required.
\end{proof}

\begin{corollary}\label{cor:def-with-intersections}
Let $M\models Th(FG)$ be sufficiently saturated, and let
$A=\acl(\bar a)$ and $B=\acl(\bar b)$ be algebraically closed subsets of
$M$. Suppose that every connected component of
$G_{\mathrm{tree}}(M)$ meeting $A$ or $B$ contains a point of
$A\cap B$. If a set $X\subseteq M^n$ is definable over both $A$ and
$B$, then $X$ is definable over $A\cap B$.
\end{corollary}

The next step is to reduce the general case to the setting of Corollary \ref{cor:def-with-intersections}.

\begin{lemma}\label{make-it-remote}
Let $M\models Th(FG)$ be sufficiently saturated, let $M_1$ be a connected
component of $M$, and let $a',b'\in M_1$ be distinct elements. Then, for
every finite tuple $\bar u$ from $M_1$ and every $k\in\mathbb N$, there
exists
\[
g\in\langle Aut(M_1/a'), Aut(M_1/b')\rangle
\]
such that
\[
d(a',g(u_j))\geq k
\]
for every coordinate $u_j$ of $\bar u$.
\end{lemma}
 
\begin{proof}
It is enough to prove that, given any finite tuple $\bar u$ from $M_1$,
there exists
\[
h\in\langle Aut(M_1/a'), Aut(M_1/b')\rangle
\]
satisfying $$d\big(a',h(u_j)\big)
=
d(a',u_j)+2d(a',b'), \text{ for every coordinate }u_j\text{ of }\bar u.$$
 
By the extension property of non-forking independence, there exists a
tuple $\bar u^{\,1}$ such that
\[
\tp(\bar u^{\,1}/a')=\tp(\bar u/a')
\quad\text{and}\quad
\indep{\bar u^{\,1}}{a'}{b'}.
\]
By saturation, there is an automorphism
$\rho\in Aut(M_1/a')$ sending $\bar u$ to $\bar u^{\,1}$.

Since $\bar u^{\,1}$ is independent from $b'$ over $a'$, every geodesic
from $u_j^{\,1}$ to $b'$ passes through $a'$. Hence
\[
d(b',u_j^{\,1})
=
d(b',a')+d(a',u_j^{\,1})
=
d(b',a')+d(a',u_j)
\qquad\text{for every }j,
\]
where the last equality follows from the fact that $\rho$ fixes $a'$.
 
Reapplying the same argument, there exists an automorphism
$\tau\in Aut(M_1/b')$ such that
$\bar u^{\,2}:=\tau(\bar u^{\,1})$ satisfies
\[
\indep{\bar u^{\,2}}{b'}{a'}.
\]
As $\tau$ fixes $b'$, we have
$d(b',u_j^{\,2})=d(b',u_j^{\,1})$ for every $j$. Moreover, by the
description of forking independence, every geodesic from $u_j^{\,2}$ to
$a'$ passes through $b'$. Therefore,
\[
d(a',u_j^{\,2})
=
d(a',b')+d(b',u_j^{\,2})
=
d(a',b')+d(b',u_j^{\,1})
=
2d(a',b')+d(a',u_j)
\]
for every $j$. Consequently,
\[
h:=\tau\circ\rho
\]
has the desired property.
\end{proof}

The following remark is a direct consequence of the quantifier
elimination provided by Fact~\ref{FGQE}. It collects two elementary
observations that play a fundamental role in the arguments that follow:
first, that the distance between any two entries of a tuple satisfying
an atomic $L'$-formula is uniformly bounded, and second, that every
$L'$-formula naturally decomposes according to the connected components
containing its parameters.

\begin{remark}\label{rem:components}
Every atomic $L'$-formula is uniformly local. More precisely, if
$\chi(z_1,\ldots,z_r)$ is an atomic $L'$-formula, where $r\leq 3$, then
there exists $N(\chi)<\omega$ such that for every model $M\models
Th(FG)$ and every tuple $c_1,\ldots,c_r\in M$,
\[
M\models\chi(c_1,\ldots,c_r)
\quad\Longrightarrow\quad
d(c_p,c_q)\leq N(\chi)
\]
for all $1\leq p,q\leq r$. Consequently, every tuple satisfying an
atomic $L'$-formula lies entirely within a single connected component of
$M$. In particular, if $a$ and $b$ belong to distinct connected
components, then
\[
M\models\neg P_\delta(a,b)
\]
and
\[
M\models\forall x\,\neg Y_\varepsilon(a,b,x).
\]

Now let $M_1$ be a connected component of $M$, let
$\bar a_1\in M_1^l$ and
$\bar a_2\in(M\setminus M_1)^m$, and let
$\theta(\bar x,\bar a_1,\bar a_2)$ be an $L'$-formula. By
Fact~\ref{FGQE}, $\theta$ is equivalent to a Boolean combination of
atomic $L'$-formulas. Rewriting this Boolean combination in
disjunctive normal form, we may therefore assume that $\theta$ is a
disjunction of conjunctions of atomic and negated atomic
$L'$-formulas. Since every atomic formula involving parameters from
both $\bar a_1$ and $\bar a_2$ is either identically false or the
negation of such a formula is identically true, after removing
inconsistent disjuncts and tautological literals, each remaining
disjunct can be written as the conjunction of a formula involving only
the parameters $\bar a_1$ and a formula involving only the parameters
$\bar a_2$. Hence $\theta(\bar x,\bar a_1,\bar a_2)$ is equivalent to a
formula of the form
\[
(R_1(\bar x,\bar a_1)\wedge R_1'(\bar x,\bar a_2))
\vee\cdots\vee
(R_n(\bar x,\bar a_1)\wedge R_n'(\bar x,\bar a_2)),
\]
where each $R_i$ and $R_i'$ is a conjunction of atomic and negated
atomic $L'$-formulas.
\end{remark}

\begin{lemma}\label{lemm:removing-disjoint-points}
Let $M\models Th(FG)$ be saturated, and let
$A=\acl(\bar a)$ and $B=\acl(\bar b)$ for some
tuples $\bar a,\bar b\in M$. Let $M_1$ be a connected component of $M$ such that $A\cap M_1=\{a'\}$ and $B\cap M_1=\{b'\}$. If $X\subseteq M^{l}$ is definable over $A$ and over $B$, then either $a'=b'$, or $X$ is definable over $A\setminus\{a'\}$ and over $B\setminus\{b'\}$.
\end{lemma}

\begin{proof}
We assume that $a'\neq b'$. Choose a finite tuple
$\bar\alpha$ from $A\setminus\{a'\}$ and an $L'$-formula
$\theta$ such that
\[
X=\theta(M,\bar\alpha,a').
\]
By Remark~\ref{rem:components},
\[
\theta(\bar x,\bar\alpha,a')
\equiv
\bigvee_{i=1}^{n}
\Bigl(
R_i(\bar x,\bar\alpha)
\wedge
R_i'(\bar x,a')
\wedge
\psi_i(\bar x)
\Bigr),
\]
where each literal of $R_i$ contains a parameter from $\bar\alpha$ but
not $a'$, each literal of $R_i'$ contains the parameter $a'$ but no
parameter from $\bar\alpha$, and $\psi_i$ is quantifier-free without
parameters.

Let $N$ be the maximum of the constants $N(\chi)$ from
Remark~\ref{rem:components}, where $\chi$ ranges over the finitely many
atomic formulas occurring in $\theta$. We call a tuple
$\bar c\in M^l$ \emph{remote} if, for every coordinate $c_j$,
\[
c_j\notin M_1
\qquad\text{or}\qquad
d(a',c_j)>N.
\]

Observe that every tuple can be moved to a remote one by an automorphism
generated by $Aut(M/a')$ and $Aut(M/b')$. Indeed, if no coordinate of
$\bar c$ belongs to $M_1$, there is nothing to prove. Otherwise,
Lemma~\ref{make-it-remote}, applied with $k=N+1$ to the subtuple of
$\bar c$ contained in $M_1$, yields an automorphism
\[
\sigma\in
\langle
Aut(M_1/a'),
Aut(M_1/b')
\rangle
\]
such that $\sigma(\bar c)$ is remote. Extending $\sigma$ by the identity
outside $M_1$, we may regard it as an element of
$\langle Aut(M/a'), Aut(M/b')\rangle$.

We now compare $X$ with the formula obtained by deleting all literals
involving $a'$ from those disjuncts in which $R_i'$ contains no positive
atomic formulas. More precisely, let
\[
I_-=
\{
\,i\le n:
R_i'
\text{ consists entirely of negated atomic formulas}
\},
\]
and define
\[
\varphi(\bar x,\bar\alpha)
:=
\bigvee_{i\in I_-}
\bigl(
R_i(\bar x,\bar\alpha)
\wedge
\psi_i(\bar x)
\bigr).
\]

We claim that $X$ and $\varphi(M,\bar\alpha)$ agree on the class of
remote tuples.

Indeed, let $\bar c\in X$ be remote. If
$R_i'(\bar c,a')$ holds, then $i\in I_-$, since any positive atomic
formula in $R_i'$ would force one of the coordinates of $\bar c$ to lie
within distance at most $N$ from $a'$, contradicting remoteness.
Therefore $\varphi(\bar c,\bar\alpha)$ holds.

Conversely, suppose that
$M\models\varphi(\bar c,\bar\alpha)$, witnessed by some
$i\in I_-$. Since every literal in $R_i'(\bar x,a')$ is the negation of
an atomic formula, there is $N_i<\omega$ such that
$R_i'(\bar c,a')$ holds whenever
$d(a',c_j)>N_i$ for every coordinate $c_j\in M_1$. As $\bar c$ is
remote, this condition is satisfied. Hence
$R_i'(\bar c,a')$ holds, and therefore
$\theta(\bar c,\bar\alpha,a')$ holds as well. Thus
$\bar c\in X$.

It follows that $X$ and $\varphi(M,\bar\alpha)$ coincide on remote
tuples.

Finally, let $\bar c\in M^l$. Choose
\[
\sigma\in
\langle
Aut(M/a'),
Aut(M/b')
\rangle
\]
such that $\sigma(\bar c)$ is remote. Since $X$ is definable over both
$A$ and $B$, it is invariant under $\sigma$. Moreover,
$\sigma(\bar\alpha)=\bar\alpha$. Therefore,
\[
\bar c\in X
\iff
\sigma(\bar c)\in X
\iff
M\models\varphi(\sigma(\bar c),\bar\alpha)
\iff
M\models\varphi(\bar c,\bar\alpha),
\]
where the middle equivalence uses the previous claim. Hence
$X=\varphi(M,\bar\alpha)$, so $X$ is definable over
$A\setminus\{a'\}$. By symmetry, it is also definable over
$B\setminus\{b'\}$.
\end{proof}

\begin{lemma}\label{lmm:every-conn-comp}
Let $M\models Th(FG)$ be sufficiently saturated, let
$A=\acl(\bar a)$ and $B=\acl(\bar b)$ for some tuples
$\bar a,\bar b\in M$, and let $X\subseteq M^n$ be definable over both
$A$ and $B$. Suppose that $M_1$ is a connected component of $M$
meeting $A$ but disjoint from $B$. Then $X$ is definable over
$A\setminus M_1$.
\end{lemma}
\begin{proof}

Put
\[
A_1=A\cap M_1
\qquad\text{and}\qquad
C=A\setminus M_1,
\]
so that
\[
A=A_1\sqcup C.
\]

It suffices to show that $X$ is invariant under
$Aut(M/C)$. Accordingly, let
$\sigma\in Aut(M/C)$. If
$\sigma\!\upharpoonright_{A_1}=\mathrm{id}_{A_1}$, then
$\sigma$ fixes $A$ pointwise, and hence
$\sigma(X)=X$. We may therefore assume that
$\sigma$ moves some point of $A_1$.

Since automorphisms preserve connected components, $\sigma(M_1)$ is
again a connected component of $M$. We distinguish two cases according
to whether $\sigma(M_1)$ meets $B$.

Suppose first that $\sigma(M_1)\cap B=\emptyset$. Since both $M_1$ and $\sigma(M_1)$ are disjoint from $B\cup C$, by
saturation and homogeneity there is an automorphism
\[
\delta_1\in Aut(M/B)
\]
which fixes every connected component meeting $B\cup C$ pointwise and
satisfies
\[
\delta_1\!\upharpoonright_{\sigma(A_1)}
=
(\sigma\!\upharpoonright_{A_1})^{-1}.
\]
Setting
\[
\delta_2=\delta_1\circ\sigma,
\]
we obtain
\[
\delta_2\in Aut(M/A).
\]
Since $X$ is definable over both $A$ and $B$, we have
\[
\delta_2(X)=X
\qquad\text{and}\qquad
\delta_1(X)=X.
\]
Therefore
\[
\sigma(X)
=
\delta_1^{-1}\delta_2(X)
=
X.
\]

It remains to consider the case
$\sigma(M_1)\cap B\neq\emptyset$. Put $M_2=\sigma(M_1)$. Since
$M_1\cap B=\emptyset$ whereas $M_2\cap B\neq\emptyset$, the components
$M_1$ and $M_2$ are distinct, and hence disjoint.

We first observe that $M_2$ is also disjoint from $A$. Indeed,
$A_1\subseteq M_1$, so $A_1\cap M_2=\emptyset$. On the other hand,
$\sigma$ fixes $C=A\setminus M_1$ pointwise. If some $c\in C$ belonged
to $M_2=\sigma(M_1)$, then $\sigma^{-1}(c)\in M_1$; but
$\sigma^{-1}(c)=c$, contradicting $c\notin M_1$. Thus
$C\cap M_2=\emptyset$, and consequently
\[
A\cap M_2=\emptyset.
\]

Since $A=\acl(\bar a)$ and $B=\acl(\bar b)$ are generated by finite
tuples, they meet only finitely many connected components of $M$.
Therefore, by saturation and homogeneity, there exist a connected
component $M_3$, disjoint from $A\cup B$, and an automorphism
$\sigma'\in Aut(M/A)$ satisfying
\[
\sigma'(M_2)=M_3.
\]

Now $\sigma'\circ\sigma$ fixes $C$ pointwise and
\[
(\sigma'\circ\sigma)(M_1)=M_3,
\]
which is disjoint from $B$. Therefore the first case applies to
$\sigma'\circ\sigma$, and we obtain
\[
\sigma'\circ\sigma(X)=X.
\]
Moreover, $\sigma'(X)=X$, since $\sigma'\in Aut(M/A)$ and $X$ is
$A$-definable. Consequently,
\[
\sigma(X)
=
\sigma'^{-1}\circ\bigl(\sigma'\circ\sigma(X)\bigr)
=
X.
\]

Thus every automorphism fixing $C=A\setminus M_1$ pointwise preserves
$X$ setwise. Hence $X$ is definable over $A\setminus M_1$, as required.

\end{proof}

For the proof of the next result, we shall refer to the unique point of
a non-empty convex subset $A$ of a tree lying on every geodesic joining
a point of $A$ to a point of a disjoint convex subset $B$ as the
\emph{gate of $A$ with respect to $B$}.

\begin{proposition}\label{IntersectionAlg}
Let $M\models Th(FG)$ be sufficiently saturated, and let
$A=\acl(\bar a)$ and $B=\acl(\bar b)$ for some tuples
$\bar a,\bar b\in M$. Every subset of $M^n$ that is definable over both
$A$ and $B$ is definable over $A\cap B$.
\end{proposition}

\begin{proof}
Let $X\subseteq M^n$ be definable over both $A$ and $B$. By
Lemma~\ref{lmm:every-conn-comp}, we may first discard every connected
component meeting exactly one of $A$ and $B$. Thus, we may assume that
every connected component of $M$ either meets both $A$ and $B$, or is
disjoint from both.

If every connected component meeting $A$ and $B$ also contains a point
of $A\cap B$, then the conclusion follows immediately from
Corollary~\ref{cor:def-with-intersections}. We may therefore assume that
there are connected components meeting both $A$ and $B$ but not
$A\cap B$.

Let $M_1,\ldots,M_k$ be precisely these components, and put
\[
A_i=A\cap M_i
\qquad\text{and}\qquad
B_i=B\cap M_i.
\]
Let $a_i$ be the gate of $A_i$ with respect to $B_i$, and let $b_i$ be
the gate of $B_i$ with respect to $A_i$. Since
$A_i\cap B_i=\emptyset$, we have $a_i\neq b_i$.

We enlarge $A$ and $B$ by adjoining the corresponding gates. For each
$i$, put
\[
A_i'=\acl(A_i\cup\{b_i\})
\qquad\text{and}\qquad
B_i'=\acl(B_i\cup\{a_i\}),
\]
and define
\[
A'
=
A\cup\bigcup_{i=1}^k A_i',
\qquad
B'
=
B\cup\bigcup_{i=1}^k B_i'.
\]

We first observe that
\[
A_i'\cap B_i=\{b_i\}.
\]
Indeed, by Fact~\ref{FGAlc}, every point of
$Conv_{\mathrm{tree}}(A_i')$ lies on a geodesic joining $b_i$ to a point
of $Conv_{\mathrm{tree}}(A_i)$. Since $b_i$ is the gate of $B_i$ with
respect to $A_i$, no such geodesic meets $B_i$ except at $b_i$. Hence
$A_i'\cap B_i=\{b_i\}$. By symmetry,
\[
A_i\cap B_i'=\{a_i\}.
\]

It follows that
\[
A'\cap B=(A\cap B)\cup\{b_1,\ldots,b_k\}
\]
and, similarly,
\[
A\cap B'=(A\cap B)\cup\{a_1,\ldots,a_k\}.
\]

Since $X$ is definable over $A$, it is also definable over $A'$, while
its definability over $B$ is unchanged. Moreover, every connected
component meeting both $A'$ and $B$ now contains a point of
$A'\cap B$: in each of the components $M_i$ this point is $b_i$, while
all the remaining components already contain a point of $A\cap B$.
Therefore, by Corollary~\ref{cor:def-with-intersections}, $X$ is
definable over
\[
(A\cap B)\cup\{b_1,\ldots,b_k\}.
\]

Applying the same argument to $A$ and $B'$, we obtain that $X$ is also
definable over
\[
(A\cap B)\cup\{a_1,\ldots,a_k\}.
\]

Finally, Lemma~\ref{lemm:removing-disjoint-points} can be applied
successively in the components $M_1,\ldots,M_k$. Since
$a_i\neq b_i$ for every $i$, it removes the two additional parameters
$a_i$ and $b_i$ one component at a time. Consequently, $X$ is definable
over $A\cap B$, as required.
\end{proof}

The following result is an immediate consequence of Propositions \ref{DCCAlg} and \ref{IntersectionAlg}.

\begin{theorem}\label{WEI}
The first-order theory of the Farey graph weakly eliminates imaginaries.
\end{theorem}

We are now ready to prove the main result of this paper. 

\begin{theorem}
$Th(FG)$ is 1-ample but not 2-ample.
\end{theorem}
\begin{proof}
We note that by weak elimination of imaginaries, it suffices to work with the real algebraic closure. 

Let $M$ be a sufficiently saturated model of $Th(FG)$, and let
$a,b\in M$ be distinct elements that belong to the same connected component. Then, by
Fact~\ref{FGFind}, $a$ forks with $b$ over $\emptyset$, while
Fact~\ref{FGAlc} yields $\acl(a)\cap\acl(b)=\emptyset$. Hence $Th(FG)$ is $1$-ample.

We now show that $Th(FG)$ is not $2$-ample. Suppose towards a contradiction that there exist tuples $\bar a_0,\bar a_1,\bar a_2 \in M$ witnessing that $Th(FG)$ is $2$-ample over some set of parameters $A$. Then the following conditions hold:
    \begin{enumerate}
        \item $acl(A\bar{a}_0\bar{a}_1) \cap acl(A\bar{a}_0 \bar{a}_2) = acl(A \bar{a}_0)$ \label{item1}
        \item $acl(A \bar{a}_0) \cap acl( A \bar{a}_1) = acl(A)$ \label{item2}
        \item  $\depen{\bar{a}_0}{A}{\bar{a}_2}$
        \label{item3}
        \item $\indep{\bar{a}_0}{A \bar{a}_1}{\bar{a}_2}$ \label{item4}
    \end{enumerate}

By \ref{item3}, there is a geodesic between $\acl(A\bar a_0)$ and
$\acl(A\bar a_2)$ that does not meet $Conv_{\mathrm{tree}}(A)$. We
choose such a geodesic, say $[p,q]$, of minimum length. By
\ref{item4}, every geodesic joining $\acl(A\bar a_0)$ and
$\acl(A\bar a_2)$ meets $Conv_{\mathrm{tree}}(A\bar a_1)$.

Let $C$ denote the connected component of $G_{\mathrm{tree}}(M)$
containing $[p,q]$. Since $[p,q]$ joins $\acl(A\bar a_0)$ to
$\acl(A\bar a_2)$, both $\acl(A\bar a_0)$ and $\acl(A\bar a_2)$ meet
$C$. Moreover, as $[p,q]$ meets
$Conv_{\mathrm{tree}}(A\bar a_1)$, the set $\acl(A\bar a_1)$ also
meets $C$. Finally, since $[p,q]$ is disjoint from
$Conv_{\mathrm{tree}}(A)$, the set $\acl(A)$ may either be disjoint from
$C$ or meet $C$ non-trivially. By Fact~\ref{FGAlc}, algebraic closure
is computed component-wise. Hence, replacing each of the algebraically
closed sets appearing in \ref{item1}--\ref{item4} by its intersection
with $C$, the four conditions continue to hold.

Let $[p_0,q_0]=[p,q]\cap Conv_{\mathrm{tree}}(A\bar a_1)$, and let
$r\in \acl(A\bar a_1)$ be the last element of $[p_0,q_0]$ when
travelling from $p$ to $q$. By \ref{item2}, and since $[p,q]$ is
disjoint from $Conv_{\mathrm{tree}}(A)$, we have
$r\notin\acl(A\bar a_0)$.

On the other hand, by \ref{item1} and
Corollary~\ref{IndAlgClosed}, the geodesic from $r$ to $q$ meets
$Conv_{\mathrm{tree}}(A\bar a_0)$. Let $s\in\acl(A\bar a_0)$ be the
first element of this intersection. Since $r\notin\acl(A\bar a_0)$,
the point $s$ lies strictly between $r$ and $q$. Therefore, the
subgeodesic $[s,q]$ is strictly shorter than $[p,q]$. As it joins
$\acl(A\bar a_0)$ to $\acl(A\bar a_2)$ and is contained in $[p,q]$, it
is still disjoint from $Conv_{\mathrm{tree}}(A)$, contradicting the
minimality of $[p,q]$.

\end{proof}

\bibliographystyle{plain}
\bibliography{biblio}

\end{document}